\documentclass[12pt]{amsart} \usepackage{amssymb,amsmath,amstext}
\usepackage[colorlinks=false, urlcolor=blue, linkcolor=blue,
citecolor=blue]{hyperref} 
\usepackage {tikz}
\usepackage {cite}
\usepackage{comment}
\usetikzlibrary{arrows,decorations.pathmorphing,backgrounds,positioning,fit,matrix}
\definecolor{processblue}{cmyk}{0.96,0,0,0} \usepackage{enumerate}
\input{xy} \xyoption{all} \usepackage{xypic}
\hypersetup{
  colorlinks   = true, 
  urlcolor     = blue, 
  linkcolor    = blue, 
  citecolor   = blue 
}
\numberwithin{equation}{section}

\theoremstyle{plain} \newtheorem{theorem}{Theorem}[section]
\newtheorem{proposition}[theorem]{Proposition} \newtheorem{lemma}[theorem]{Lemma}
\newtheorem{corollary}[theorem]{Corollary}

\theoremstyle{definition} \newtheorem{definition}[theorem]{Definition}

\DeclareMathOperator{\Hess}{Hess}\DeclareMathOperator{\rk}{rk}
\DeclareMathOperator{\ann}{Ann}
\DeclareMathOperator{\Cat}{Cat}\DeclareMathOperator{\B}{\mathcal{B}}
\DeclareMathOperator{\K}{\sf k}
\begin{document}

\author{Nasrin Altafi} \address{Department of Mathematics, KTH Royal
  Institute of Technology, S-100 44 Stockholm, Sweden}
\email{nasrinar@kth.se}

\title[Hilbert functions of Artinian Gorenstein algebras with the SLP]{Hilbert functions of Artinian Gorenstein algebras with the strong Lefschetz property}

\maketitle

\begin{abstract}
We prove that a sequence $h$ of non-negative integers is the Hilbert function of some Artinian Gorenstein algebra with the strong Lefschetz property if and only if it is an SI-sequence. This generalizes the result by T. Harima which characterizes the Hilbert functions of Artinian Gorenstein algebras with the weak Lefschetz property. We also provide classes of Artinian Gorenstein algebras obtained from the ideal of points in $\mathbb{P}^n$ such that some of their higher Hessians have non-vanishing determinants. Consequently, we provide families of such algebras satisfying the SLP.
\end{abstract}

\section{Introduction}
An Artinian graded algebra $A$ over a field $\K$ is said to satisfy the weak Lefschetz property (WLP for short) if there exists a linear form $\ell$ such that the multiplication map $\times \ell:A_i\rightarrow  A_{i+1}$ has maximal rank for every $i\geq 0$. An algebra $A$ is said to satisfy the strong Lefschetz property (SLP) if there is a linear form $\ell$ such that $\times \ell^j:A_i\rightarrow  A_{i+j}$ has maximal rank for each $i,j\geq 0$. Determining which graded Artinian algebras satisfy the Lefschetz properties has been of great interest (see for example \cite{HMNW, MMN2, Gondim, BK, Lefbook, BMMNZ, tour} and their references).
It is known that every Artinian algebra of codimension two in characteristic zero has the SLP, it was proven many times using different techniques, see for example \cite{HMNW} and \cite{Briancon}. This is no longer true for codimension three and higher and in general it is not easy to determine Artinian algebras satisfying or failing the WLP or SLP. Studying the Lefschetz properties of Artinian Gorenstein algebras is a very interesting problem. The $h$-vector of an Artinian algebra with the WLP is unimodal.  In general there are examples of Artinian Gorenstein algebras with non-unimodal $h$-vector and hence failing the WLP. R. Stanley \cite{Stanley} gave the first example with $h$-vector  $h=(1,13,12,13,1)$. Later D. Bernstein and A. Iarrobino \cite{BI} and M. Boij and D. Laksov \cite{BL} provided examples of non-unimodal Gorenstein $h$-vector with $h_1=5$.
Sequence $h=\left(h_0,h_1,\dots \right)$ is a Stanley-Iarrobino sequence, or briefly SI-sequence, if it is symmetric, unimodal and its first half, $(h_0,h_1,\dots , h_{\lfloor\frac{d}{2}\rfloor})$ is differentiable.
R. Stanley \cite{Stanley} showed that the Hilbert functions of Gorenstein sequences  are SI-sequences for $h_1\leq 3$. By the examples of non-unimodal Gorenstein Hilbert functions it is known that this is not necessarily true for $h_1\geq 5$.  Whether Hilbert functions of Artinian Gorenstein algebras with $h_1=4$ are SI-sequences is still open. It is known that any SI-sequence is a Gorenstein $h$-vector \cite{ChoIarrobino, MiglioreNagel}. T. Harima in \cite{Harima1995} gave a characterization on $h$-vectors of Artinian Gorenstein algebras satisfying the WLP. In this article we generalize this result and characterize $h$-vectors of Artinian Gorenstein algebras satisfying the SLP, see Theorem \ref{SI-SLP-Theorem}.\par
In section \ref{section4}, we consider classes of Artinian Gorenstein algebras which are quotients of coordinate rings of a set of $\K$-rational points in $\mathbb{P}_{\K}^n$. We prove that for a set $X$ of points in $\mathbb{P}_{\K}^n$ which lie on a rational normal curve any Artinian Gorenstein quotient of $A(X)$ satisfies the SLP, Theorem \ref{smoothconic}. 
Higher Hessians of dual generators of Artinian Gorenstein algebras were introduced by T. Maeno and J. Watanabe \cite{MW}. We study the higher Hessians of dual generators of Artinian Gorenstein quotients of $A(X)$. We show  Artinian Gorenstein quotients of $A(X)$ where  $X\subset \mathbb{P}_{\K}^2$ lie on a conic satisfy the SLP, Theorem \ref{singularConic}. 
We also prove non-vanishing of the determinants of certain higher Hessians  in  Theorems \ref{points-on-conic} and  \ref{points-on-line}  for Artinian Gorenstein quotients of coordinate ring of points $X\subset \mathbb{P}_{\K}^2$ where $X$ contains points on a conic and a line respectively. We then in Corollary \ref{corSLP} provide classes of such Artinian algebras satisfying SLP.
\section{Preliminaries}

Let $S={\sf k}[x_0,\dots ,x_n]$ be a polynomial ring equipped with the standard grading over a field $\sf k$ of characteristic zero and $\mathbb{P}^{n}=\mathbb{P}^{n}_{\sf k}=\mathrm{Proj} S$.  Let $A = S/I$ be a graded Artinian (it has Krull dimension zero) algebra where $I$ is a homogeneous ideal. The \emph{Hilbert function} of $A$ in degree $i$ is $h_{A}(i) = h_i =\dim_{\sf k}(A_i)$. Since $A$ is Artinian the Hilbert function of $A$ is determined by its $h$-vector, $h=\left(h_0,h_1,h_2,\dots ,h_d\right)$ such that $h_d\neq 0$. The integer $d$ is called the \emph{socle degree}. The graded ${\sf k}$-algebra is \emph{Gorenstein} if it has a one dimensional socle. Without loss of generality we may assume that $I$ does not contain a linear form (form of degree $1$) so $h_1=n+1$ and is called the codimension of $A$. If $A$ is Gorenstein then the $h$-vector is symmetric and so $h_d=1$. A sequence $h=\left(h_0,\dots ,h_d\right)$ is called a \emph{Gorenstein sequence} if $h$ is the Hilbert function of some Artinian Gorenstein algebra. \par
\noindent Let $h$ and $i$ be positive integers. Then $h$ can be written uniquely in the following form 
\begin{equation}
h=\binom{m_i}{i}+\binom{m_{i-1}}{i-1}+\cdots +\binom{m_j}{j},
\end{equation}
where $m_i > m_{i-1}>\cdots > m_j\geq j\geq 1$. This expression for $h$ is called the $i$-binomial expansion of $h$. Also define 
\begin{equation}
h^{\langle i\rangle}=\binom{m_i+1}{i+1}+\binom{m_{i-1}+1}{i}+\cdots +\binom{m_j+1}{j+1}
\end{equation}
where we set $0^{\langle i\rangle}:=0$.

A sequence of non-negative integers $h=\left(h_0,h_1,\dots \right)$ is called an \emph{O-sequence} if $h_0=1$ and $h_{i+1}\leq h_i^{\langle i\rangle}$ for all $i\geq 1$. Such sequences are the ones which exactly occur as Hilbert functions of standard graded algebras.
\begin{theorem}[Macaulay\cite{Macaulay}]
The sequence $h = \left(h_0,h_1,\dots , h_d\right)$ is an O-sequence if and only if it is the $h$-vector of some standard graded Artinian algebra. 
\end{theorem}
We say $h=\left(h_0,h_1,\dots \right)$ is  \emph{differentiable} if its first difference $\Delta h = \left(h_0,h_1-h_0,\dots \right)$ is an O-sequence. Moreover, an $h$-vector is called \emph{unimodal} if $h_0\leq h_1\leq\cdots\leq  h_i \geq  h_{i+1}\geq \cdots \geq h_d$. A sequence $h=\left(h_0,h_1,\dots \right)$ is \emph{Stanley-Iarrobino sequence}, or briefly \emph{SI-sequence}, if it is symmetric, unimodal and its first half, $(h_0,h_1,\dots , h_{\lfloor\frac{d}{2}\rfloor})$ is differentiable.\par  

Now we recall the theory of \emph{Macaulay Inverse systems}. Define Macualay dual ring $R= {\sf k}[X_0,\dots ,X_n]$ to $S$ where the action of $x_i$ on $R$, which is denoted by $\circ$, is partial differentiation with respect to $X_i$.
For a homogeneous ideal $I\subseteq S$ define its \emph{inverse system} to be the graded $S$-module $M\subseteq R$ such that $I=\ann_S(M)$. 
There is a one-to-one correspondence between graded Artinian algebras $S/I$ and finitely generated graded $S$-submodules $M$ of $R$, where $I=\ann_S(M)$ is the annihilator  of $M$ in $S$, conversely, $M=I^{-1}$ is the $S$-submodule of $R$ which is annihilated by  $I$. Moreover, the Hilbert functions of $S/I$ and $M$ are the same, in fact $\dim_{\sf k}(S/I)_i=\dim_{\sf k}M_i$ for all $i\geq 0$. See \cite{Geramita} and \cite{IK} for more more details. 

By a  result by F.H.S. Macaulay \cite{F.H.S} it is known that an Artinian standard graded $\mathsf{k}$-algebra $A=S/I$ is Gorenstein if and only if there exists $F\in R_d$, such that $I=\ann_S(F)$. The homogeneous polynomial $F\in R_d$ is called the \emph{Macaulay dual generator} of $A$.
\begin{definition}\cite[Definition 3.1]{MW}
Let $F$ be a polynomial in $R$ and $A= S/\ann_S(F)$ be its associated Artinian Gorenstein algebra. Let $\mathcal{B}_{j} = \lbrace \alpha^{(j)}_i+\ann_S(F) \rbrace_i$ be a  $\mathsf{k}$-basis of $A_j$. The entries of the  $j$-th Hessian matrix of $F$ with respect to $\mathcal{B}_j$ are given  by
$$
(\Hess^j(F))_{u,v}=(\alpha^{(j)}_u\alpha^{(j)}_v \circ F).
$$
We note that when $j=1$ the form $\Hess^1(F)$ coincides with the usual Hessian. Up to  a non-zero constant multiple  $\det \Hess^j(F)$ is independent of the basis $\mathcal{B}_j$.  By abuse of notation we will write   $\mathcal{B}_{j} = \lbrace \alpha^{(j)}_i \rbrace_i$ for a basis of $A_j$. For a linear form $\ell=a_0x_0+\cdots +a_nx_n$ we denote by $\Hess^j_\ell(F)$ the Hessian evaluated at the point $P$ dual to $\ell$ that is $P=(a_0,\dots ,a_n)$.
\end{definition}
The following result by T. Maeno and J. Watanabe provides a criterion for Artinian Gorenstein algebras satisfying the SLP.
\begin{theorem}\cite[Theorem 3.1]{MW}
Let $A = S/\ann_S(F)$ be an Artinian Gorenstein quotient of $S$ with socle degree $d$. Let $\ell$ be a linear form and consider the multiplication map $\times \ell^{d-2j} :A_j\longrightarrow A_{d-j}$. Pick  any bases $\mathcal{B}_j$ for $A_j$ for $j=0,\dots , \lfloor\frac{d}{2}\rfloor$. Then linear form $\ell$ is a strong Lefschetz element for $A$ if and only if 
$$
\det\Hess^j_\ell(F)\neq 0,
$$
for every $j=0,\dots , \lfloor\frac{d}{2}\rfloor$.
\end{theorem}
\begin{definition}
 Let $A= S/\ann(F)$ where $F\in R_d$. Pick bases $\B_j = \lbrace \alpha^{(j)}_u\rbrace_u$ and $\B_{d-j} = \lbrace \beta^{(d-j)}_u\rbrace_u$ be $\sf k$-bases of $A_j$ and $A_{d-j}$ respectively. The entries of the  \emph{catalecticant matrix of $F$} with respect to $\B_j$ and $\B_{d-j}$ are given by 
$$
(\Cat^j_F)_{uv}=(\alpha^{(j)}_u\beta^{(d-j)}_v\circ F).
$$
\end{definition}
Up to a non-zero constant multiple $\det \Cat^j_F$ is independent of the basis $\mathcal{B}_j$. The rank of the $j$-th catalecticant matrix of $F$ is equal to the Hilbert function of $A$ in degree $j$, see \cite[Definition 1.11]{IK}.

Throughout this paper we denote by  $X=\{P_1,\dots ,P_s\}$ a set of $s$ distinct  points in $\mathbb{P}^n$. Denote the coordinate ring of $X$ by $A(X)=S/I(X)$, where $I(X)$ is a homogeneous ideal of forms vanishing on $X$. For each point $P\in X$ we consider it as a point in the affine space $\mathbb{A}^{n+1}$ and denote it by $P=(a_0 : \dots  : a_n)$ and the linear form in $R$ dual to $P$ is $L=a_0X_0+\cdots +a_nX_n$. We set 
\begin{equation}
\tau(X):=\min \{i\mid h_{A(X)}(i)=s\}.
\end{equation}
A. Iarrobino and V. Kanev \cite{IK} proved that any Artinian Gorenstein quotient of $A(X)$ by a general enough  hyperplane  of degree $d\geq \tau(X)$ has the Hilbert function that is equal to $h_{A(X)}$ in degrees $0\leq j\leq \lfloor\frac{d}{2}\rfloor$.

M. Boij in \cite{Boij} provides the special form of the dual generator for an Artinian Gorenstein quotients of $A(X)$.
\begin{proposition}\cite[Proposition 2.3]{Boij}\label{dualForm}
Let $A$ be any Artinian Gorenstein quotient of $A(X)$ with socle degree $d$ such that $d\geq \tau(X)$ and dual generator $F$. Then $F$ can be written as 
$$
F= \sum^s_{i=1}\alpha_iL_i^d
$$
where $\alpha_1,\dots ,\alpha_s\in \K$, are not all zero.
\end{proposition}
\begin{proposition}\cite[Proposition 2.4]{Boij}\label{h-vector}
Assume that $d\geq 2\tau(X)-1$ and $F =\sum^s_{i=1}\alpha_iL_i^d$ where $\alpha_i\neq 0$ for all $i$. Let $A$ be the Artinian Gorenstein quotient of $A(X)$ with dual generator $F$. Then the Hilbert function of $A$ is given by 
$$
h_A(i)=\begin{cases}
&h_{A(X)}(i)  \quad\quad\quad\hspace*{0.25cm} 0\leq i\leq \lfloor \frac{d}{2}\rfloor,\\
&h_{A(X)}(d-i)\quad \quad \lceil\frac{d}{2}\rceil\leq i\leq d.
\end{cases}
$$
\end{proposition}
The following well known result  guarantees the existence of a set of points $X\subseteq\mathbb{P}^n$ with given Hilbert function under an assumption on the Hilbert function.

\begin{theorem}\cite[Theorem 4.1]{GMR} \label{diffOseqthm}
Let $h=(h_0,h_1,\dots )$ be a sequence of non-negative integers. Then
 there is a reduced $\sf k$-algebra with Hilbert function $h$ if and only if $h$ is a differentiable 
 O-sequence.
\end{theorem} 
The following theorem due to E. D. Davis \cite{Davis} provides information about the geometric properties of a set of points $X\subseteq \mathbb{P}^2$ given the Hilbert function $h_{A(X)}$.
\begin{theorem}\cite{Davis}\label{davis}
Let $X\subseteq \mathbb{P}^2$ be a set of distinct points such that 
$\Delta h_{A(X)} = (\mathrm{h}_0,\mathrm{h}_1,\dots ,\mathrm{h}_{\tau(X)}).$ Assume that $\mathrm{h}_j=\mathrm{h}_{j+1}=r$ for some $j\geq t$ where $t$ is the smallest degree of the generators of the defining ideal of $X$. Then $X$ is a disjoint union of $X_1\subseteq \mathbb{P}^2$ and $X_2\subseteq \mathbb{P}^2$ such that $X_1$ lies on a curve of degree $r$ and $\Delta h_{A(X_2)}= (\mathrm{h}_r-r,\mathrm{h}_{r+1}-r	,\dots ,\mathrm{h}_{j-1}-r)$.
\end{theorem}

\section{Hilbert functions of Artinian Gorenstein algebras and SI-sequences}
In this section we give a  characterization of the Hilbert functions of Artinian Gorenstein algebras satisfying the SLP which generalizes Theorem 1.2 in \cite{Harima1995}.
We do so by using the higher Hessians of the Macaulay dual form of Artinian Gorenstein algebras. We first provide  an explicit expression for the higher Hessians of such polynomials $F=\sum_{i=1}^s\alpha_iL_i^d$.
\begin{lemma}\label{HessLemma}
Let $A$ be an Artinian Gorenstein quotient of $A(X)$ with dual generator $F = \sum^s_{i=1}\alpha_iL_i^d$ where $\alpha_i\neq 0$ for all $i$ and $d\geq 2\tau(X)-1$. Then for each $0\leq j\leq \tau(X)-1$ we have that 
\begin{equation}\label{hess}
\det\Hess^j(F) = \sum_{\mathcal{I}\subseteq \{1,\dots ,s\}, \vert\mathcal{I}\vert=h_{A(X)}(j)} c_\mathcal{I}\prod_{i\in \mathcal{I}}\alpha_iL_i^{d-2j},
\end{equation}
where $c_{\mathcal{I}}\in \K$. \par 
\noindent Moreover, $c_{\mathcal{I}}\neq 0$ if and only if for $X_{\mathcal{I}} = \{P_i\}_{i\in \mathcal{I}}$ we have that $h_{A(X)}(j)=h_{A(X_{\mathcal{I}})}(j)$.
\end{lemma}
\begin{proof}
We have that 
$$
\Hess^j(F) = \sum^s_{i=1}\alpha_i\Hess^j(L_i^d).
$$
Notice that $\Hess^j(L_i^d)$ is a rank one matrix that is equal to $L_i^{d-2j}$ times a scalar matrix. Let  $T=\K[\alpha_1, \dots ,\alpha_s, L_1, \dots , L_s]$ be a polynomial ring over $\K$. For each $P_i = (a_{i,0}: \dots : a_{i,n})\in X$ denote  $L_i=a_{i,0}X_0+\cdots+a_{i,n}X_n$ and define the action of $S$ on $T$ by $x_j\circ L_i=a_{i,j}$ and $x_j\circ \alpha_i=0$ for every $1\leq i\leq s$ and $0\leq j\leq n$. \par \noindent We consider  $\det \Hess^j(F)$ as a bihomogeneous polynomial in $T$ having bidegree $\left(h_{A(X)}(j), (d-2j)h_{A(X)}(j)\right)$.
We claim that $\det \Hess^j(F)$ is square-free in $\alpha_i$'s. 
We prove the claim by showing that the coefficient of any monomial in $T$ that has exponent larger than one in $\alpha_i$'s is zero. Without loss of generality, we let $\alpha_1$ to be the only one that has exponent two. So we show that $\alpha^2_1L^{2(d-2j)}_1\prod^{h_{A(X)}(j)-1}_{i=2}\alpha_iL_i^{d-2j}$, which has bidegree $\left(h_{A(X)}(j), (d-2j)h_{A(X)}(j)\right)$, has zero coefficient in $\det \Hess^j(F)$. Assume not and set 
\begin{equation}\label{lemeq}
\det\left(\Hess^j(F)\biggm\vert_{\alpha_{h_{A(X)}(j)}=\cdots =\alpha_s=0}\right)= \det\left(\sum^{h_{A(X)}(j)-1}_{i=1}\alpha_i\Hess^j(L_i^d) \right) = \lambda \alpha^2_1L^{2(d-2j)}_1\prod^{h_{A(X)}(j)-1}_{i=2}\alpha_iL_i^{d-2j}\neq 0
\end{equation}
 for some $\lambda\in \mathsf{k}^*$. Notice that  $\Hess^j(F)\biggm\vert_{\alpha_{h_{A(X)}(j)}=\cdots =\alpha_s=0}$ is a square matrix of size $h_{A(X)}(j)$ and by the above equation has maximal rank. On the other hand, $\sum^{h_{A(X)}(j)-1}_{i=1}\alpha_i\Hess^j(L_i^d)$ is the sum of ${h_{A(X)}(j)-1}$ rank one matrices which has rank at most equal to $h_{A(X)}(j)-1$ that is a contradiction.\par
\noindent Now let $\mathcal{I}\subseteq \{1,\dots ,s\}$ such that $\vert\mathcal{I}\vert=h_{A(X)}(j)$. If $c_{\mathcal{I}}\neq 0$ in Equation (\ref{hess}) setting $\alpha_i=0$ for every $i\in \{1,\dots , s\}\setminus \mathcal{I}$ implies that $\det\Hess^j(F)\neq 0$ and therefore  $h_{A(X)}(j)=h_{A(X_{\mathcal{I}})}(j)$. Conversely, assume that we have $h_{A(X)}(j)=h_{A(X_{\mathcal{I}})}(j)$ then since $\vert\mathcal{I}\vert=h_{A(X)}(j)$ we pick  $\mathcal{B}_j=\{L^j_i\}_{i\in\mathcal{I}}$ as a basis for  $A(X)_j$. Therefore, setting  $\alpha_i=0$ for every $i\in \{1,\dots , s\}\setminus \mathcal{I}$ implies that $\Hess^j(F)$ with respect to $\mathcal{B}_j$ is a diagonal matrix with diagonal entries equal to $\frac{d!}{(d-2j)!}\alpha_iL^{d-2j}_i$ for every $i\in \mathcal{I}$ which implies that $c_{\mathcal{I}}\neq 0$.
\end{proof}
Now we are able to state and prove the main result of this section.
\begin{theorem}\label{SI-SLP-Theorem}
Let $h=\left( h_0,h_1,\dots ,h_d\right)$ be a sequence of positive integers. Then $h$ is the Hilbert function of some Artinian Gorenstein algebra with the SLP if and only if $h$ is an SI-sequence.
\end{theorem}
\begin{proof}
Suppose $A$ is an Artinian Gorenstein algebra with the Hilbert function $h$ and the strong Lefschetz element $\ell\in A_1$  that is in particular the weak Lefschetz element and using  \cite[Theorem 1.2]{Harima1995} we conclude that  $h$ is an SI-sequence.

Conversely, assume that $h$ is an SI-sequence. We set $h_1=n+1$ and $h_t=s$ where $t=\min \{i\mid h_i\geq h_{i+1}\}$. So we have 
\begin{equation}\label{hilbertfunction}
h=\left(1,n+1,\dots ,s,\dots , s,\dots ,n+1,1\right).
\end{equation}
Define a sequence of integers $\overline{h}=(\overline{h}_0,\overline{h}_1,\dots )$ such that $\overline{h}_i=h_i$ for $i=0,\dots , t$ and $\overline{h}_i=s$ for $i\geq t$. Assuming that $h$ is an SI-sequence implies that $\overline{h}$ is a differentiable O-sequence and by Theorem \ref{diffOseqthm} there exists  $X = \{P_1,\dots ,P_s\}\subseteq \mathbb{P}^{n}$ such that the Hilbert function of its coordinate ring $A(X)$ is equal to $\overline{h}$, that is $h_{A(X)}=\overline{h}$.  Denote by  $\{L_1,\dots, L_{s}\}$ the linear forms dual to $\{P_1,\dots ,P_s\}$.
 As in Proposition \ref{dualForm}, let $A$ be the Artinian Gorenstein quotient of $A(X)$ with dual generator  $F = \sum^{s}_{i=1}\alpha_i L^{d}_i$ for $d\geq 2\tau(X)$, notice that $\tau(X)=t$. By Proposition \ref{h-vector}, in order to have  $h_A=h$ we must have $\alpha_i\neq 0$ for all $i$. Let $L$ be a linear form dual to $P\in \mathbb{P}^n$ such that $X\cap P=\emptyset$ and denote by $\ell$ the dual linear form to $L$. Therefore, $\beta_i:= \ell\circ L_i\neq 0$ for every $1\leq i\leq s$. We claim that there exist $\alpha_1, \dots , \alpha_{s}$ such that $\ell$ is the strong Lefschetz element for $A$. First note that for every $j = t,\dots ,\lfloor\frac{d}{2}\rfloor$ the multiplication map by $\times \ell^{d-2j}:A_j\rightarrow A_{d-j}$ can be considered as the multiplication map on $A(X)$, that is $\times \ell^{d-2j}:A(X)_j\rightarrow A(X)_{d-j}$ which has trivially maximal rank.\par 
Now we prove that there is a Zariski open set for $\alpha_i$'s such that for every $j=0,\dots , t-1$ the $j$-th Hessian matrix of $F$ evaluated at $\ell$ has maximal rank, that is 
\begin{align*}
\rk \Hess^j_\ell(F)= h_A(j)=h_j.
\end{align*}
Using Lemma \ref{HessLemma} we get that 
\begin{equation}
\det \Hess^j_\ell(F)=\sum_{\mathcal{I}\subseteq \{1,\dots ,s\},\vert\mathcal{I}\vert=h_j}c_\mathcal{I}\prod_{i\in \mathcal{I}}\alpha_i\beta^{d-2j}_i
\end{equation}
where $c_\mathcal{I}\neq 0$ if and only if $h_{A(X_{\mathcal{I}})}(j)=h_j$ for $X_{\mathcal{I}}=\{P_i\}_{\mathcal{I}}$. Notice that since there is at least one subset $\mathcal{I}$ such that $c_{\mathcal{I}}\neq 0$ the determinant of the $j$-th Hessian is not identically zero. Therefore, $\det \Hess^j_\ell\neq 0$ or equivalently $\rk\Hess^j_\ell(F)=h_j$ for each $i=0,\dots ,t-1$, provides a Zariski open subset of $\mathbb{P}^{s-1}$ for $\alpha_i$'s and therefore the intersection of all those open subsets is non-empty. Equivalently, there is an
Artinian Gorenstein algebra $A$ such that $h_A=h$ and satisfies the SLP with $\ell\in A_1$.
\end{proof}
\section{Higher Hessians of Artinian Gorenstein quotients of $A(X)$ }\label{section4}
In this section we prove the non-vanishing of some of the higher Hessians for any Artinian Gorenstein quotient of $A(X)$ for $X\subset \mathbb{P}^n$ under some conditions on the configuration of the points in $X$. In some cases we conclude that they satisfy the SLP.
\begin{proposition}\label{tophess-s-1}
Let $X=\{P_1,\dots , P_s\}$ be a set of points in $\mathbb{P}^n$ and $A$ be any Artinian Gorenstein quotient of $A(X)$ with dual generator $F=\sum_{i=1}^s\alpha_iL_i^d$ for $d\geq 2\tau(X)-1$. Assume that $h_A(j)=s-1$ for some $j\geq 0$ then there is a linear form $\ell$ such that 
$$
\det\Hess^j_\ell(F)\neq 0.
$$
\end{proposition}
 \begin{proof}
 Using Lemma \ref{detLemma} we have that
 $$
\det   \Hess^j(F) = \sum_{\mathcal{I}\subseteq \{1,\dots ,s\}, \vert\mathcal{I}\vert=s-1}c_\mathcal{I}\prod_{i\in \mathcal{I}}\alpha_iL^{d-2j}_i.
  $$
We prove that $\det \Hess^j(F)\neq 0 $ as a polynomial in $X_i$'s. 
Without loss of generality assume that for $\mathcal{I}=\{2,\dots , s \}$ we have that $c_{\{2,\dots ,s\}}\neq 0$. Then if  $\det\Hess^j(F)$ is identically zero  we get that 
$$
c_{\{2,\dots , s\}}\prod^s_{i=2}\alpha_iL^{d-2j}_i = -\alpha_1L^{d-2j}_1\left(\sum_{\mathcal{I}\subseteq\{2,\dots ,s\}, \vert\mathcal{I}\vert=s-2}c_\mathcal{I}\prod_{i\in \mathcal{I}}\alpha_iL^{d-2j}_i\right).
$$
This contradicts the fact that $R={\sf k}[X_0,\dots,X_n]$ is a unique factorization domain. We conclude that there exists $\ell$ such that   $\det \Hess_\ell^j(F)\neq 0$.
\end{proof}
\begin{proposition}\label{tophess-s-2}
Let $s\geq 3$ and $X=\{P_1,\dots , P_s\}$  be a set of points in $\mathbb{P}^2$ in a general linear position.
Let $A$ be any Artinian Gorenstein quotient of $A(X)$ with dual generator $F=\sum_{i=1}^s\alpha_iL_i^d$ for $d\geq 2\tau(X)-1$ and assume that $h_A(j)=s-2$ for some $0 \leq j\leq \frac{d+1}{2}$. Then there is a linear form $\ell$ such that 
$$
\det\Hess^j_\ell(F)\neq 0.
$$
 \end{proposition}
 \begin{proof}
If $h_A(j+1)<s-2$ then the maximum value of $h_A$ is equal to $s-2$ and also we have that $j=\frac{d}{2}$. Therefore,  $\Hess^j(F)=\Cat^j_F$ is the trivial multiplication on $A_j$ which clearly has maximal rank.\par 
\noindent If $h_A(j+1)>s-2$ then the assumption on $X$ implies that $h_A(j+1)=s$. Since we have $h_A(j+1)=s-1$ the last three non-zero entries of $\Delta h_{A(X)}$ are equal to one. So E. D. Davis's theorem \ref{davis} implies that $X$ contains at least three colinear points. 
Therefore, 
 $$
 h_A=(1,3,\dots , s-2,\underbrace{s,\dots ,s}_k, s-2, \dots , 3,1),
 $$
 for some $k\geq 1$. Note that for a linear form $\ell$ such that $\ell\circ L_i\neq 0$ for every $i$ we get that the multiplication map $\ell^{d-2i}:A_i\rightarrow A_{d-i}$ for every $j+1\leq i\leq \lfloor\frac{d}{2}\rfloor$ is a map on $A(X)$ in the same degrees and therefore has maximal rank. So $\det\Hess^j_\ell(F)\neq 0$ if and only if $\det\Hess^j_\ell(\ell^k\circ F)\neq 0$.
 Denote by $\beta_i=\ell\circ L_i\neq 0$ for each $i$. So we have that
 $$
G:= \ell^k\circ F = \frac{d!}{(d-k)!}\sum^s_{i=1}\alpha_i\beta^k_iL^{d-k}_i.
 $$
The Artinian Gorenstein quotient of $A(X)$ with dual generator $G$ has the following Hilbert function
$$
(1,3,\dots ,s-2,s-2,\dots ,3,1).
$$
Therefore, it is enough to show that $\det\Hess_\ell^j(F)\neq 0$ for some $\ell$ in the case $h_A(j)=h_A(j+1)=s-2$. Note that in this case $d=2j+1$. By Lemma \ref{detLemma} we have that 
  \begin{equation}
 \det \Hess^j(F) = \sum_{\mathcal{I}\subseteq \{1,\dots ,s\}, \vert\mathcal{I}\vert=s-2}c_\mathcal{I}\prod_{i\in \mathcal{I}}\alpha_iL_i,
 \end{equation}
such that $c_{\mathcal{I}}\neq 0$ if and only if $h_{A(X)}(j)=h_{A(X_{\mathcal{I}})}(j)$. Without loss of generality assume that $c_{\{3,\dots ,s\}}\neq 0$. Suppose that  $\det \Hess^{j}(F)=0$ then 
$$
c_{\{3,\dots , s\}}\prod^s_{i=3}\alpha_iL_i = -\alpha_1L_1\left(\sum_{\mathcal{I}\subseteq\{3,\dots ,s\}, \vert\mathcal{I}\vert=s-3}c_\mathcal{I}\prod_{i\in \mathcal{I}}\alpha_iL_i\right) -\alpha_2L_2\left(\sum_{\mathcal{I}\subseteq\{1,3,\dots ,s\}, \vert\mathcal{I}\vert=s-3}c_\mathcal{I}\prod_{i\in \mathcal{I}}\alpha_iL_i\right).
$$
Common zeros of $L_1$ and $L_2$ correspond to the line passing through $P_{1}$ and $P_{2}$. By the assumption this line does not pass through any other point in $\{P_{3},\dots ,P_{s}\}$ which means that the left hand side of the above equality is nonzero on the points where $L_1=L_2=0$ which is a contradiction. So this implies that $\det \Hess^j(F)\neq 0$ and therefore $\det\Hess_\ell^{j}(F) \neq 0$,  for some linear form $\ell$.
 \end{proof}

 \begin{theorem}[Points on a rational normal curve]\label{smoothconic}
Let $X=\{P_1,\dots , P_s\}$ be a set of points in $\mathbb{P}^n$ lying on a rational normal curve. Assume that $A$ is an Artinian Gorenstein quotient of $A(X)$ with dual generator $F=\sum_{i=1}^s\alpha_iL_i^d$ for $d\geq 2\tau(X)$ and $\alpha_i\neq 0$ for every $i$. Then $A$ satisfies the SLP.
\end{theorem}
\begin{proof}
Denote by $Y =\{Q_1,\dots ,Q_s\}\subset \mathbb{P}^1$ the preimage  of $X$ under the Veronese embedding $\varphi :\mathbb{P}^1\longrightarrow \mathbb{P}^n.$
For each $i=1,\dots ,s$ denote by $K_i$ the linear form in ${\sf k}[S,T]$ dual to $Q_i$. Let $B$ be any Artinian Gorenstein quotient of  the coordinate ring of $Y$, ${\sf k}[s,t]/I(Y)$, with dual generator $G = \sum_{i=1}^s\beta_iK_i^{nd}$. \par 
\noindent The Artinian Gorenstein algebra $B$ has the following Hilbert function
$$
h_B=(1,2,3,\dots, \underbrace{s,\dots , s}_k,\dots ,3,2,1),
$$
for some $k\geq 1$ since we have assumed that $d\geq 2\tau(X)$ and $\alpha_i\neq 0$ for every $i$.  It is known that  $B$ has the SLP  for some linear form $\ell\in B_1$ \cite[Proposition 2.2]{HMNW}. The Veronese embedding $\varphi$ gives a map of rings $\psi : S=\mathsf{k}[x_0,\dots ,x_n]\rightarrow \mathsf{k}[s,t]$ defined by taking each $x_i$ to a different monomial of degree $n$ in $s,t$. The map $\psi$ induces isomorphisms $A_j\cong B_{nj}$ as $\sf k$-vector spaces for every $j$. Let $\ell^\prime:=\psi^{-1}(\ell^n)\in A_1$, then we have 
$$\rk\left(\times (\ell^\prime)^{d-2j}:A_j\longrightarrow A_{d-j} \right) = \rk\left(\times (\ell)^{n(d-2j)}:B_{nj}\longrightarrow B_{nd-nj} \right) = nj+1=\dim_{\sf k}A_j.$$
Thus $A$ satisfies the  SLP with  linear form $\ell^\prime$.
\end{proof}

The above proposition shows that every Artinian Gorenstein quotient of $A(X)$ such that $X\subset \mathbb{P}^2$ consists of points on a smooth conic  satisfies the SLP.
We will show that the SLP also holds when $X\subset \mathbb{P}^2$ consists of points on a singular  conic.  \par 
\noindent First we need to prove a lemma.
\begin{lemma}\label{detLemma}
Let $A=B+C$ be  a square matrix of size $2m-1$ for $m\geq 1$ as the following
\begin{small}
\begin{equation}
B = \begin{pmatrix}
f_1&f_2&\dots & f_m &0&\cdots &0\\
f_2&f_3&\dots & f_{m+1}&0&\cdots &0\\
\vdots &\vdots && \vdots  &\vdots & &\vdots \\
f_{m} &f_{m+1} &\dots & f_{2m}  &0&\cdots &0\\
0&0&\dots & 0 &0&\cdots &0\\
\vdots &\vdots && \vdots  &&\vdots &\vdots \\
0&0&\dots & 0  &0&\cdots &0\\
\end{pmatrix}, \quad C = \begin{pmatrix}
0&\dots & 0  &0&\cdots &0&0\\
0&\dots & 0  &0&\cdots &0&0\\
\vdots & &\vdots & \vdots  &&\vdots &\vdots \\
0 &\dots &0& g_{2m}  &\dots & g_{m+1} &g_m\\
\vdots &&\vdots & \vdots  &&\vdots &\vdots \\
0 &\dots &0& g_{m+1}  &\dots & g_{3} &g_2\\
0 &\dots &0& g_{m}  &\dots & g_{2} &g_1\\
\end{pmatrix}.
\end{equation}
\end{small}

Then $$
\det A = (\det B_{\{1,\dots , m-1\}})(\det C_{\{m,\dots ,2m-1\}})+(\det B_{\{1,\dots , m\}})(\det C_{\{m+1,\dots ,2m-1\}}),
$$
such that for a subset $\mathcal{J}\subset \{1,\dots ,2m-1\}$ we denote by $B_{\mathcal{J}}$ and $C_{\mathcal{J}}$ the square submatrices of $B$ and $C$ respectively with rows and columns in the index set $\mathcal{J}$.
\end{lemma}
\begin{proof}
We have that 
$$
\det A = \sum_{\sigma\in S_{2m-1}}\mathrm{sign}\sigma A_{1\sigma_1}\dots A_{(2m-1)\sigma_{2m-1}},
$$
where the entry $A_{i\sigma_i}$ is the entry in row $i$ and column $\sigma_i$. Then we split $\det A$ in the following way
\begin{align*}
\det A =&(\sum_{\sigma\in S_{m-1}}\mathrm{sign}\sigma A_{1\sigma_1}\dots A_{(m-1)\sigma_{m-1}}) A_{m,m}(\sum_{\tau\in S_{m-1}}\mathrm{sign}\tau A_{(m+1)(m+\tau_{1})}\dots A_{(2m-1)(m+\tau_{m-1})})\\
&+ (\sum_{\sigma\in S_{m}, \sigma_m\neq m}\mathrm{sign}\sigma A_{1\sigma_1}\dots A_{m\sigma_{m}}) (\sum_{\tau\in S_{m-1}}\mathrm{sign}\tau A_{(m+1)(m+\tau_{1})}\dots A_{(2m-1)(m+\tau_{m-1})})\\
&+ (\sum_{\sigma\in S_{m-1}}\mathrm{sign}\sigma A_{1\sigma_1}\dots A_{(m-1)\sigma_{m-1}}) (\sum_{\tau\in S_{m}, \tau_1\neq 1}\mathrm{sign}\tau A_{m(m-1+\tau_{1})}\dots A_{(2m-1)(m-1+\tau_{m})})\\
=&(\det B_{\{1,\dots , m-1\}})(\det C_{\{m,\dots ,2m-1\}})+(\det B_{\{1,\dots , m\}})(\det C_{\{m+1,\dots ,2m-1\}}).
\end{align*}
\end{proof}
\begin{theorem}\label{singularConic}
Assume that $X=\{P_1,\dots ,P_s\}$ is  a set of points in $\mathbb{P}^2$ which lie on a conic. Let $A$ be an Artinian Gorenstein quotient of $A(X)$ with dual generator $F=\sum_{i=1}^s\alpha_iL_i^d$, for $d\geq 2\tau(X)$ and $\alpha_i\neq0$ for every $i$. Then $A$ satisfies the SLP.
\end{theorem}
\begin{proof}
If $X$ lies on a smooth conic applying Theorem \ref{smoothconic} for $n=2$ we get the desired result. 
Now suppose that  $X$ consists of points on a singular conic that is a union of two lines in $\mathbb{P}^2$. Suppose that $X_1:=\{P_1,\dots , P_{s_1}\}$ is a subset of $X$ which lies on one line and $X_2:=\{Q_{1},\dots , Q_{s_2}\}$ is a subset of $X$ with the points on the other line, so $X=X_1\cup X_2$. If $X_1\cap X_2=\emptyset$ then $s_1+s_2=s$ otherwise $s_1+s_2-1=s$. Denote by $L_i$ the linear form dual to $P_i$ for $1\leq i\leq s_1$ and by $K_i$ the linear form dual to $Q_i$ for each $1\leq i\leq s_2$. Let $F_1 =\sum^{s_1}_{i=1} a_iL_i^d $ and $F_2= \sum^{s_2}_{i=1} b_iK_i^d$ for linear forms $L_i$ and $K_i$ where $F=F_1+F_2$. By linear change of coordinates we may assume that $L_i=u_{0,i}X_0+u_{2,i}X_2$ and  $K_i=v_{1,i}X_1+ v_{2,i}X_2$ such that $u_{0,i},u_{2,i},v_{0,i},v_{2,i}\in \sf k$ for every $i$. The Hilbert function of $A$ is equal to 
$$
h_A = \left(1,3,5,\dots, 2k+1,s,\dots ,s, 2k+1, \dots , 5, 3,1\right),
$$
where $k$ is the largest integer such that $2k+1\leq s$. If $s=2k+1$ then $\tau(X)=k$ and otherwise $\tau(X)=k+1$. Let $j$ be an integer such that $1\leq j\leq \tau(X)-1$.
 Consider the following ordered monomial basis for $A$ in degree $j$ 
$$\B_j=\{x_0^j,x_0^{j-1}{x_2},\dots , x_0x_2^{j-1},x_2^j,x_2^{j-1}x_1,\dots , x_2x_1^{j-1},x_1^{j} \}.$$ The $j$-th Hessian of $F$ with respect to $\B_j$ is the following matrix
\begin{align*}
\Hess^j(F) &= \Hess^j(F_1)+\Hess^j(F_2) = \sum^{s_1}_{i=1} a_i\Hess^j(L_i)+\sum^{s_2}_{i=1} b_i\Hess^j(K_i)\\
& = \begin{pmatrix}
C^j_0&C^j_1&\dots & C^j_j &\cdots &0&0\\
C^j_1&C^j_2&\dots & C^j_{j+1}&\cdots &0&0\\
\vdots &\vdots && \vdots  &&\vdots &\vdots \\
C^j_j & C^j_{j+1} &\dots & C^j_{2j}+D^j_{2j}  &\dots & D^j_{j+1} &D^j_j\\
\vdots &\vdots && \vdots  &&\vdots &\vdots \\
0&0&\dots & D^j_{j+1}  &\dots & D^j_{2} &D^j_1\\
0&0&\dots & D^j_{j}  &\dots & D^j_{1} &D^j_0\\
\end{pmatrix}
\end{align*}
where we set $C^j_i = ({x_0^{j-i}x_2^{i}})\circ  F_1$ and $D^j_i =({x_1^{j-i}x_2^{i}})\circ F_2$ for each $i=0,\dots ,j$.\par 
\noindent Then using Lemma \ref{detLemma} we get that
\begin{equation}\label{HessDecomposition}
\det \Hess^j(F) = (\det C^j_{\{0,\dots ,j-1\}}) (\det D^j) +(\det D^j_{\{1,\dots ,j\}}) (\det C^j),
\end{equation}
where we set $C^j = \begin{pmatrix}
C^j_0&C^j_1&\dots & C^j_j\\
C^j_1&C^j_2&\dots & C^j_{j+1}\\
\vdots &\vdots && \vdots \\
C^j_j & C^j_{j+1} &\dots & C^j_{2j}
\end{pmatrix}
$ and $D^j = \begin{pmatrix}
D^j_{2j}&D^j_{2j-1}&\dots & D^j_j\\
D^j_{2j-1}&D^j_{2j-2}&\dots & D^j_{j-1}\\
\vdots &\vdots && \vdots \\
D^j_j & D^j_{j-1} &\dots & D^j_{0}
\end{pmatrix}$ and we denote by  $C^j_{\{i_1,\dots ,i_r\}}$ the  square submatrix of $C^j$ of size $r$ with rows and columns $i_1,\dots ,i_r$, similarly for $D^j$. \par 
\noindent  Let $A_1$ and $A_2$ be Artinian Gorenstein quotients  of $A(X_1)={\sf k}[x_0,x_2]/I(X_1)$ and $A(X_2) = {\sf k}[x_1,x_2]/I(X_2)$ with dual generators $F_1$ and $F_2$ respectively. We observe that $C^j = \Hess^j(F_1)$ and $D^j = \Hess^j(F_2)$. Since every Artinian algebra of codimension two has the SLP we have that $\det C^j\neq 0$ and $\det D^j\neq 0$.\par 
\noindent We set 
\begin{align*}
F^\prime_1 :=  x_0^2\circ F_1, \quad F^\prime_2 := x_1^2\circ F_2.
\end{align*}
Then $C^j_{\{0,\dots ,j-1\}}$ is equal to the $(j-1)$-th Hessian of $F^\prime_1$ with respect to the ordered basis $\{x_0^{j-1}, x_0^{j-2}x_2,\dots ,x_2^{j-1}\}$. Similarly, $D^j_{1,\dots , j} = \Hess^{j-1}(F^\prime_2)$ with respect to \begin{small}
$\{x_2^{j-1}, x_2^{j-2}x_1,\dots ,x_1^{j-1}\}$.
\end{small} So using the result that Artinian algebras in codimension two have the SLP we get that 
$$\det\Hess^{j-1}(F^\prime_1) = \det C^j_{\{0,\dots ,j-1\}}\neq 0,\quad \text{and}\quad \det \Hess^{j-1}(F^\prime_2) = \det D^j_{\{1,\dots ,j\}}\neq 0.\quad $$ 
Therefore, Equation (\ref{HessDecomposition}) is equivalent to 
\begin{equation}
\det \Hess^j(F) = (\det \Hess^{j-1}(F^\prime_1)) (\det \Hess^j(F_2)) +(\det \Hess^{j-1}(F^\prime_2)) (\det \Hess^j(F_1)).
\end{equation}
Note that assuming $d\geq 2\tau(X)$ and $1\leq j\leq \tau(X)-1$ implies that 
$$\deg(\det \Hess^{j-1}(F^\prime_1))<\deg (\det \Hess^j(F_1)),\quad \deg(\det \Hess^{j-1}(F^\prime_2))<\deg (\det \Hess^j(F_2)).$$
Therefore, $\det\Hess^j(F)\neq0$ unless when $X_2^{d-2j}$ is a factor of both $\det\Hess^j(F_1)$ and $\det\Hess^j(F_2)$ so we must have $X_1\cap X_2\neq \emptyset$ and  $s=s_1+s_2-1$. On the other hand, using Lemma \ref{HessLemma} we get that $\det\Hess^j(F_1)$ is in fact a non-zero monomial in $L_i$'s and $j=s_1$. Similarly, we get $j=s_2$. So $s=2s_1-1=2s_2-1=2k+1$ and therefore $\tau(X)=k$ and $s_1=s_2=k+1=\tau(X)+1$. This contradicts the assumption that $j\leq \tau(X)-1$.\par 
For each $\tau(X)\leq j\leq \lfloor\frac{d}{2}\rfloor$ the $j$-the Hessian of $F$ corresponds to the multiplication map on $A(X)$ and then trivially has maximal rank for general enough linear forms. \par Note that $\det\Hess^0(F)=F\neq 0$. Therefore, we have proved that there is a linear form $\ell$ such that $\det\Hess_\ell^j(F)\neq 0$ for every $0\leq j\leq \lfloor\frac{d}{2}\rfloor$ and equivalently $A$ has the SLP.
\end{proof}
We now prove that if $X\subseteq \mathbb{P}^2$ contains points on a conic then higher Hessians of $F$ of high enough order are non-zero. First we set a notation that for every $i\geq 0$ the subscript of the entry $ \Delta h_A (i) = h_i-h_{i-1}$  is denoted by $i$.
\begin{theorem}\label{points-on-conic}
Let $X=\{P_1,\dots ,P_s\}$ be a set of points in $\mathbb{P}^2$ and $A$ be an Artinian Gorenstein quotient of $A(X)$ with dual generator $F=\sum_{i=1}^s\alpha_iL_i^d$, for $d\geq 2\tau(X)$ and $\alpha_i\neq 0$ for every $i$. Suppose that the first difference of $h_A$ is equal to 
$$\Delta h_A = (1,2,h_2-3,\dots , 2_k,\dots ,2_{\tau(X)}),$$
for some $1\leq k< \tau(X)$. Then there is a linear form $\ell$ such that for every $k-1\leq j\leq \lfloor\frac{d}{2}\rfloor$
$$\det\Hess^j_\ell(F)\neq 0.$$

\end{theorem}
\begin{proof}
Since $\Delta h_{A(X)}$ is flat, the Theorem \ref{davis} due to E. D. Davis \cite{Davis} implies that $X$ is a disjoint union of $2\tau(X)+1$ points on a conic and $s-2\tau(X)-1$ other points.  We may assume that $P_1,\dots , P_{s-2\tau(X)-1}$ lie outside the conic. Note that for each $k-1\leq j\leq \lfloor\frac{d}{2}\rfloor$ we have that $h_{A(X)}(j) = 2j+1+s-2\tau(X)-1=s-2\tau(X)+2j$.  Using Lemma \ref{HessLemma} we get that for each $k-1\leq j\leq \lfloor\frac{d}{2}\rfloor$
\begin{equation}\label{iff}
\det \Hess^j(F) =\sum_{\mathcal{I}\subseteq{\{1,\dots ,s\}}, \vert\mathcal{I}\vert = h_{A(X)}(j)}c_{\mathcal{I}}\prod_{i\in\mathcal{I}}\alpha_iL_i^{d-2j},
\end{equation}
where  $c_{\mathcal{I}}\neq0$ if and only if $h_{A(X_\mathcal{I})}(j)=h_{A(X)}(j) = s-2\tau(X)+2j$.  Notice that the Hilbert function of the coordinate ring of the points on a conic in degree $j$ is at most $2j+1$. Therefore, $c_{\mathcal{I}}\neq 0$ if and only if $\mathcal{I}$ contains $s-2\tau(X)+2j-(2j+1) = s-2\tau(X)-1$  points off the conic that means $\{1,\dots ,s-2\tau(X)-1\}\subset \mathcal{I}$. \par 
\noindent This implies that $\prod_{i=1}^{s-2\tau(X)-1}\alpha_iL^{d-2j}_i$ is a common factor of the right hand side of Equation (\ref{iff}), so 
\begin{equation}\label{factoriff}
\det \Hess^j(F) = \prod_{i=1}^{s-2\tau(X)-1}\alpha_iL^{d-2j}_i\left(\sum_{\mathcal{I}\subseteq{\{s-2\tau(X),\dots ,s\}}, \vert\mathcal{I}\vert = 2j+1}c_{\mathcal{I}}\prod_{i\in\mathcal{I}}\alpha_iL_i^{d-2j}\right).
\end{equation}
Let $Y := \{P_{s-2\tau(X)},\dots ,P_s\}$ be the subset of $X$ which lies on a conic. Consider the Artinian Gorenstein quotient of $A(Y)$ with dual generator $G=\sum_{i=s-2\tau(X)}^{s}\alpha_iL_i^d$. 
Theorem \ref{singularConic} implies that $B$ has the SLP. Equivalently, for every $0\leq j\leq \lfloor\frac{d}{2}\rfloor$
$$\det\Hess^j(G) = \sum_{\mathcal{I}\subseteq{\{s-2\tau(X),\dots ,s\}}, \vert\mathcal{I}\vert = 2j+1}c_{\mathcal{I}}\prod_{i\in\mathcal{I}}\alpha_iL_i^{d-2j}\neq 0.$$
 This implies that the polynomial in Equation (\ref{factoriff}) is non-zero and this completes the proof.
\end{proof}
Similarly, using that all Artinian algebras in codimension two have the SLP we have the following which proves non-vanishing of some of higher Hessians in the case where $X\subseteq \mathbb{P}^2$ contains points on a line. 
\begin{theorem}\label{points-on-line}
Let $X=\{P_1,\dots ,P_s\}$ be a set of points in $\mathbb{P}^2$ and $A$ be an Artinian Gorenstein quotient of $A(X)$ with dual generator $F=\sum_{i=1}^s\alpha_iL_i^d$, for $d\geq 2\tau(X)$ and $\alpha_i\neq 0$ for every $i$. Suppose that the first difference of $h_A$ is equal to 
$$\Delta h_A = (1,2,h_2-3,\dots , 1_k,\dots ,1_{\tau(X)}),$$
for some $1\leq k<\tau(X)$. Then there is a linear form $\ell$ such that for every $k-1\leq j\leq \lfloor\frac{d}{2}\rfloor$
$$\det\Hess^j_\ell(F)\neq 0.$$
\end{theorem}
\begin{proof}
Using Theorem \ref{davis} we get that there are exactly $\tau(X)+1$ points on a line and $s-\tau(X)-1$ off the line. We may assume that $P_1,\dots , P_{s-\tau(X)-1}$ lie off the line. For each $k-1\leq j\leq \lfloor\frac{d}{2}\rfloor$ we have $h_{A(X)}(j)=j+1+s-\tau(X)-1=s-\tau(X)+j $.\par 
\noindent So for each $s-k\leq j\leq \lfloor\frac{d}{2}\rfloor$ by Lemma \ref{detLemma} we get
\begin{equation}\label{iff1}
\det \Hess^j(F) =\sum_{\mathcal{I}\subseteq{\{1,\dots ,s\}}, \vert\mathcal{I}\vert = h_{A(X)}(j)}c_{\mathcal{I}}\prod_{i\in\mathcal{I}}\alpha_iL_i^{d-2j},
\end{equation}
where  $c_{\mathcal{I}}$ is non-zero if and only if $h_{A(X_\mathcal{I})}(j) =h_{A(X)}(j)=s-\tau(X)+j$.  Since the Hilbert function of the coordinate ring of the points on a line in degree $j$ is at most $j+1$, in order for the coordinate ring of $\{P_i\}_{i\in \mathcal{I}}$ to have the Hilbert function equal to $s-\tau(X)+j$ in degree $j$, $\mathcal{I}$ must contain all the indices from $1$ to $s-\tau(X)+j-(j+1) = s-\tau(X)-1$. \par 
\noindent This implies that
\begin{equation}\label{factoriff1}
\det \Hess^j(F) =\prod_{i=1}^{s-\tau(X)-1}\alpha_iL^{d-2j}_i\left(\sum_{\mathcal{I}\subseteq{\{s-\tau(X),\dots ,s\}}, \vert\mathcal{I}\vert = j+1}c_{\mathcal{I}}\prod_{i\in\mathcal{I}}\alpha_iL_i^{d-2j}\right).
\end{equation}
Denote by $Y:=\{P_{s-\tau(X)},\dots ,P_s\}$ the points in $X$ which lie on a line. Consider the Artinian Gorenstein quotient of $A(Y)$ with dual generator $G=\sum_{i=s-\tau(X)}^{s}\alpha_iL_i^d$ and denote it by $B$. Since $B$ is an Artinian algebra of codimension two it satisfies the SLP.  
 Equivalently, for every $0\leq j\leq \lfloor\frac{d}{2}\rfloor$
$$\det\Hess^j(G) = \sum_{\mathcal{I}\subseteq{\{s-\tau(X),\dots ,s\}}, \vert\mathcal{I}\vert = j+1}c_{\mathcal{I}}\prod_{i\in\mathcal{I}}\alpha_iL_i^{d-2j}\neq 0$$
 This implies that $\det\Hess^j(F)\neq 0$ for every $k-1\leq j\leq \lfloor\frac{d}{2}\rfloor$.
\end{proof}
As a consequence of Theorems \ref{points-on-conic} and \ref{points-on-line} we provide a family of Artinian Gorenstein quotients of $X\subseteq \mathbb{P}^2$ satisfying the SLP. 
\begin{corollary}\label{corSLP}
Let $X=\{P_1,\dots ,P_s\}$ be a set of points in $\mathbb{P}^2$ and $A$ be any Artinian Gorenstein quotient of $A(X)$ with dual generator $F=\sum_{i=1}^s\alpha_iL_i^d$, for $d\geq 2\tau(X)$. Then $A$ satisfies the SLP if $\Delta h_A$ is  equal to one the following vectors 
\begin{align}\label{11}
(1,2,\underbrace{1,\dots ,1}_m), \quad (1,2,2,\underbrace{1,\dots ,1}_m),\quad (1,2,3,\underbrace{1,\dots ,1}_m),
\end{align}
\begin{align}\label{22}
(1,\underbrace{2,\dots ,2}_m),\quad (1,2,3,\underbrace{2,\dots ,2}_m),
\end{align}
for some $m\geq 2$. 
\end{corollary}
\begin{proof}
First we note that $\det \Hess^0(F) = F$ and since $F$ is assumed to be non-zero for a generic $\ell$ we have $\det \Hess_\ell^0(F) \neq 0 $. A well known result by P. Gordan and M. Noether \cite{GN} implies that the Hessian of every form in the polynomial ring with three variables is non-zero. Therefore, for a generic linear form $\ell$  we get that $\det \Hess_\ell^1(F)\neq 0$. \par 
\noindent Using Theorems Theorem \ref{points-on-line} and \ref{points-on-conic} for the  first difference vectors given in (\ref{11}) and  (\ref{22}) respectively we conclude that $\det \Hess_\ell^j(F)\neq 0$ for every $2\leq j\leq \lfloor\frac{d}{2}\rfloor $ and a generic linear form $\ell$. This completes the proof.
\end{proof}
\subsection*{Summary} We end the section by summarizing what we have shown.  Let $X=\{P_1, \dots , P_s\}\subseteq \mathbb{P}^n$ and $F=\sum_{i=1}^s\alpha_iL_i^d$, for $d\geq 2\tau(X)$ and $\alpha_i\neq 0$ for every $i$. For $n\geq 2$ if $X\subseteq \mathbb{P}^n$ lies on a rational normal curve then any Artinian Gorenstein  quotient of $A(X)$ with dual generator $F$ satisfies the SLP, Theorem \ref{smoothconic}. This result is more general for $n=2$. In fact, if $X\subseteq \mathbb{P}^2$ lies on a conic (smooth or singular) then in Theorem \ref{singularConic} we prove that any Artinian Gorenstein quotient of  $A(X)$ with dual generator $F$ satisfies the SLP. When $X\subseteq \mathbb{P}^2$, we show in Theorems \ref{points-on-conic} and \ref{points-on-line} that if the first difference of an Artinian Gorenstein quotient of $A(X)$ with dual generator $F$ is equal to $$\Delta h_A = (1,2,h_2-3,\dots , 1_k,\dots ,1_{\tau(X)}),\quad \text{or}\hspace*{2mm}\Delta h_A = (1,2,h_2-3,\dots , 2_k,\dots ,2_{\tau(X)})$$
for some $1\leq k<\tau(X)$, then there is a linear form $\ell$ such that $\det\Hess^j_\ell(F)\neq 0$ for every $k-1\leq j\leq \lfloor\frac{d}{2}\rfloor$. As a consequence of these results we show that any Artinian Gorenstein quotient $A$  of $A(X)$ with with $\Delta h_A$ given in (\ref{11}) and (\ref{22})   satisfies the SLP, Corollary \ref{corSLP}. \par 
We also show in Proposition \ref{tophess-s-1} that for every $n\geq 2$ if the $j$-th Hilbert function of an Artinian Gorenstein quotient of $A(X)$ is equal to $s-1$, that is $h_A(j)=s-1$, then $\det\Hess^j_\ell(F)\neq 0$ for some $\ell$. Also for $X\subseteq \mathbb{P}^2$ in a general linear position we have that $\det\Hess^j_\ell(F)\neq 0$ for some $\ell$ if $h_A(j)=s-2$, Proposition \ref{tophess-s-2}.

\section{Acknowledgment}
The author would like to thank Mats Boij for useful and  insightful comments and discussion that greatly assisted this research. Computations using the algebra software Macaulay2 \cite{13} were essential to get the
ideas behind some of the proofs. This work was supported by the grant VR2013-4545. 
\bibliography{bib.bbl}{}
\bibliographystyle{plain}
\end{document}